\documentclass{amsart}

\usepackage[mathscr]{eucal}
\usepackage{amssymb}
\usepackage[usenames,dvipsnames]{color}

\usepackage{amsthm}
\usepackage{bbold}
\usepackage{enumerate}
\usepackage{array}
\usepackage{amsmath}

\usepackage[colorlinks=true,linkcolor={Brown},citecolor={Brown},urlcolor={Brown}]{hyperref}

\numberwithin{equation}{section}
\usepackage[all]{xy}
\SelectTips{cm}{}

\newdir{ >}{{}*!/-10pt/\dir{>}}

\swapnumbers 

\newtheorem*{Thm*}{Theorem}
\newtheorem*{MainThm*}{Main Theorem}

\newtheorem{Lem}[equation]{Lemma}
\newtheorem{Prop}[equation]{Proposition}
\newtheorem{Thm}[equation]{Theorem}

\theoremstyle{remark}
\newtheorem{Def}[equation]{Definition}

\newtheorem{Exa}[equation]{Example}
\newtheorem{Rem}[equation]{Remark}

\newtheorem*{Que*}{Question}

\newcommand{\nc}{\newcommand}
\nc{\dmo}{\DeclareMathOperator}

\dmo{\Aut}{Aut}
\dmo{\chara}{char}
\dmo{\Der}{D}
\dmo{\End}{End}
\dmo{\HH}{H}
\dmo{\Hom}{Hom}
\dmo{\Phom}{PHom}
\dmo{\Id}{Id}
\dmo{\Ind}{Ind}
\dmo{\Infl}{Infl}
\dmo{\inj}{inj}
\dmo{\Kom}{K}
\dmo{\modname}{mod}%
\dmo{\Mod}{Mod}
\dmo{\Obj}{Obj}
\dmo{\Dim}{Dim}
\dmo{\opname}{op}
\dmo{\pr}{pr}
\dmo{\proj}{proj}
\dmo{\Qcoh}{Qcoh}
\dmo{\Rad}{Rad}
\dmo{\Res}{Res}
\dmo{\smallb}{b}
\dmo{\stabname}{stmod}
\dmo{\Stabname}{StMod}
\dmo{\tr}{tr}

\nc{\adjto}{\rightleftarrows}
\nc{\AK}{A\MModcat{K}}
\nc{\bbZ}{\mathbb{Z}}
\nc{\CC}{\mathcal{C}}
\nc{\CI}{\mathcal{I}}
\nc{\CJ}{\mathcal{J}}
\nc{\calO}{\mathcal{O}}
\nc{\cat}[1]{\mathscr{#1}}
\nc{\Db}{\Der^{\smallb}}
\nc{\Displ}{\displaystyle}
\nc{\Endcat}[1]{\End_{\cat #1}}
\nc{\eps}{\epsilon}
\nc{\Homcat}[1]{\Hom_{\cat #1}}
\nc{\hook}{\hookrightarrow}
\nc{\ie}{{\sl i.e.}\ }
\nc{\into}{\mathop{\rightarrowtail}}
\nc{\inv}{^{-1}}
\nc{\isotoo}{\overset{\sim}{\,\too\,}}
\nc{\isoto}{\overset{\sim}{\,\to\,}}
\nc{\Jon}[1]{{\color{Green}#1}}
\nc{\Kb}{\Kom^{\smallb}}
\nc{\kk}{\Bbbk}
\nc{\okk}{\overline{\kk}}
\nc{\Mid}{\,\big|\,}
\nc{\MModcat}[1]{\MMod_{\cat #1}}%
\nc{\mmod}{\text{-}\modname}%
\nc{\MMod}{\text{-}\Mod}%
\nc{\onto}{\mathop{\twoheadrightarrow}}
\nc{\op}{^{\opname}}
\nc{\otoo}[1]{\overset{#1}{\,\too\,}}
\nc{\oto}[1]{\overset{#1}\to}
\nc{\Paul}[1]{{\color{Violet}#1}}
\nc{\potimes}[1]{^{\otimes #1}}
\nc{\pproj}{\,\text{-}\proj}%
\nc{\ptimes}[1]{^{\times #1}}
\nc{\Radcat}[1]{\Rad_{\cat #1}}
\nc{\Rat}{\Rad^{\otimes}}
\nc{\Ratcat}[1]{\Rat_{\cat #1}}
\nc{\restr}[1]{_{|_{\scriptstyle #1}}}
\nc{\SET}[2]{\big\{\,#1\Mid#2\,\big\}}
\nc{\smat}[1]{\left(\begin{smallmatrix} #1 \end{smallmatrix}\right)}
\nc{\Sn}{\mathfrak{S}_n}
\nc{\Spmo}{\mathfrak{S}_{p-1}}
\nc{\sstab}{\text{-}\stabname}%
\nc{\SStab}{\text{-}\Stabname}%
\nc{\then}{$\Rightarrow$}
\nc{\too}{\mathop{\longrightarrow}\limits}
\nc{\unit}{\mathbb{1}}

\begin{document}


\title[Separable rings in stable category over cyclic~$p$-groups]
{Separable commutative rings in the stable module category of cyclic groups}
\author{Paul Balmer}
\author{Jon F. Carlson}
\date{\today}

\address{Paul Balmer, Mathematics Department, UCLA, Los Angeles,
CA 90095-1555, USA}
\email{balmer@math.ucla.edu}
\urladdr{http://www.math.ucla.edu/$\sim$balmer}

\address{Jon Carlson, Department of Mathematics, University of Georgia,
  Athens, GA 30602, USA}
\email{jfc@math.uga.edu}
\urladdr{http://www.math.uga.edu/~jfc}

\begin{abstract}
We prove that the only separable commutative ring-objects in
the stable module category of a finite cyclic $p$-group~$G$
are the ones corresponding to subgroups of~$G$.
We also describe the tensor-closure of the
Kelly radical of the module category and of the stable module
category of any finite group.
\end{abstract}


\thanks{First author was supported by NSF grant~DMS-1600032 and
Research Award of the Humboldt Foundation. The second author was
partially supported by NSA grant H98230-15-1-0007 and by Simons
Foundation grant 054813-01.}

\maketitle

\tableofcontents

\section*{Introduction}


Since 1960 and the work of Auslander and
Goldman~\cite{AuslanderGoldman60},
an algebra~$A$ over a commutative ring~$R$ is called
\emph{separable} if $A$ is projective as an $A\otimes_R A\op$-module.
This notion turns out to be remarkably important in many other contexts,
where the module category
$\cat C=R\MMod$ and its tensor $\otimes=\otimes_R$ are replaced by an arbitrary tensor
category~$(\cat C,\otimes)$. A ring-object $A$ in such a
category~$\cat C$ is \emph{separable} if multiplication
$\mu:A\otimes A\to A$ admits a section $\sigma:A\to A\otimes A$
as an $A$-$A$-bimodule in~$\cat C$. See details in Section~\ref{se:basics}.
Our main result (Theorem~\ref{thm:main}) concerns itself with modular
representation theory of finite groups:
\begin{MainThm*}
Let $\kk$ be a separably closed field of characteristic~$p>0$ and let
$G$ be a cyclic $p$-group. Let $A$ be a commutative and separable
ring-object in the stable category~$\kk G \sstab$ of finitely
generated $\kk G$-modules modulo projectives. Then there exist
subgroups $H_1,\ldots ,H_r\le G$ and an isomorphism of
ring-objects $A\simeq \kk(G/H_1)\times \cdots \times \kk(G/H_r)$.
(The ring structure on the latter is recalled below.)
\end{MainThm*}

Separable and commutative ring-objects are particularly interesting
in tensor-triangulated categories, like the above stable module
category $\kk G\sstab$.
There are several reasons for this. First, from the
theoretical perspective, if $\cat K$ is a tensor-triangulated
category (called \emph{tt-category} for short) and if $A$ is a separable
and commutative ring-object in~$\cat K$ (called \emph{tt-ring} for short)
then the category $\AK$\,, of $A$-modules in~$\cat K$, remains a
tt-category. See details in~\cite{Balmer11}. On the other hand,
from the perspective of applications,
tt-rings actually come up in many examples. Let us remind the reader.

In algebraic geometry, given an \emph{\'etale} morphism $f:Y\to X$
of noetherian and separated schemes, the object
$A=\textrm{R}f_*(\calO_Y)$ is a tt-ring in $\Der(X)=\Der(\Qcoh(X))$,
the derived
category of~$X$. Moreover, the category of $A$-modules in~$\Der(X)$
is equivalent to the derived category of~$Y$, as a tt-category.
This result is proved in~\cite{Balmer14pp}. Shortly thereafter, and it
is an additional motivation for the present paper, Neeman
proved that these ring-objects $\textrm{R}f_*(\calO_Y)$, together
with obvious localizations, are
the only tt-rings in the derived category~$\Der(X)$. The precise
statement is Theorem~7.10 in~\cite{Neeman15pp}. In colloquial terms,
the only tt-rings which appear in
algebraic geometry come from the \'etale topology.

In view of the above, one might ask: What is the analogue of
the ``\'etale topology" in modular representation theory? This
investigation was started in~\cite{Balmer15}. Let $\kk$ be a field and $G$ a
finite group, and consider $X$ a finite $G$-set. Then the
permutation $\kk G$-module $A=\kk X$ admits a multiplication
$\mu:A\otimes A\to A$ defined by $\kk$-linearly extending the
rule $\mu(x\otimes x)=x$ and $\mu(x\otimes x')=0$ for all
$x\neq x'$ in~$X$; its unit $\kk\to \kk X$ maps $1$ to~$\sum_{x\in X}x$.
This commutative ring-object $A=\kk X$ in~$\kk G\mmod$ is
separable (use $\sigma(x)=x\otimes x$), and consequently gives
a commutative separable ring-object in any tensor category which
receives $\kk G\mmod$ via a tensor functor. Hence, we inherit tt-rings
$\kk X$ in the derived category~$\Db(\kk G\mmod)$ and in the stable
module category of~$\kk G$, which is both the additive
quotient~$\kk G\sstab=\kk G\mmod/\kk G\pproj$ and the Verdier
(triangulated) quotient $\Db(\kk G\mmod)/\Kb(\kk G\pproj)$.

Since finite $G$-sets are disjoint unions of $G$-orbits and since
$\kk(X\sqcup Y)\simeq \kk X \times \kk Y$ as rings, we can focus
attention on tt-rings associated to subgroups~$H\le G$ as
\[
A^G_H:=\kk(G/H)\,.
\]
Here is an interesting fact established
in~\cite{Balmer15} about this tt-ring $A^G_H$.
Let us denote by $\cat K(G)$ either the bounded derived category
$\cat K(G)=\Db(\kk G\mmod)$, or the stable category
$\cat K(G)=\kk G\sstab$, or any variation removing the ``boundedness"
or ``finite dimensionality" conditions. Then the category of $A^G_H$-modules
in~$\cat K(G)$ is equivalent as a tt-category to the corresponding
category~$\cat K(H)$ for the subgroup~$H$:
\[
A^G_H\MModcat{K(G)}\simeq \cat K(H)\,.
\]
This description of restriction to a subgroup $\cat K(G)\to \cat K(H)$ as an
`\'etale extension' in the tt-sense is not specific to linear representation
theory but holds in a variety of equivariant settings, from topology
to C*-algebras, as shown
in~\cite{BalmerDellAmbrogioSanders15}.

\smallbreak

We hope the above short survey motivates the reader for the study of tt-rings,
and we now focus mostly on the stable category~$\cat K(G)=\kk G\sstab$.
In~\cite[Question~4.7]{Balmer15}, the first author asked whether
the above examples are the only ones:

\begin{Que*}
Let $\kk$ be a separably closed field and $G$ a finite group. Let
$A$ be a tt-ring (\ie separable and commutative) in the
stable category~$\kk G\sstab$. Is there a finite $G$-set $X$ such that
$A\simeq \kk X$ in~$\kk G\sstab$?
\end{Que*}

Equivalently, one might ask: Given a tt-ring~$A$ in~$\kk G\sstab$ which
is \emph{indecomposable} as a ring, must we have that $A\simeq\kk(G/H)$ for
some subgroup~$H\leq G$? Less formally, this is asking whether ``the
\'etale topology in modular representation theory" is completely
determined by the subgroups of~$G$, or whether some exotic
tt-rings can appear. Our Main Theorem solves this problem for cyclic
$p$-groups.

Some comments are in order. First, the reason to assume $\kk$
separably closed is obvious: If $L/\kk$ is a finite separable field
extension, then one can consider $L$ as a trivial $\kk G$-module, and
it surely defines a tt-ring in~$\kk G\sstab$ that is indecomposable as a
ring but that has really very little to
do with the group~$G$ itself. Similarly, we focus on the finite-dimensional
$\kk G$-modules, to avoid dealing with (right) Rickard idempotents as
explained in Remark~\ref{rem:idemp}.

We point out that the
answer to the above Question is positive if $\kk G\sstab$ is replaced
by the abelian category of~$\kk G$-modules
(see~\cite[Rem.\,4.6]{Balmer15}).
If $\cat C$ is the category of $\kk$-vector spaces over a field~$\kk$,
then the only commutative and separable $A\in\cat C$ are the finite
products $L_1\times \cdots \times L_n$ of finite separable field
extensions $L_1,\ldots,L_n$ of~$\kk$.
See~\cite[\S\,II.2]{DeMeyerIngraham71} or~\cite[\S\,1]{Neeman15pp}.
In particular, if we assume~$\kk$ separably closed, this ring is simply
$\kk\times\cdots\times\kk$. Remembering the action of~$G$ on the corresponding set of idempotents
is how the result is proved for~$\kk G\MMod$ in~\cite[Rem.\,4.6]{Balmer15}.
For the derived category~$\Der(\kk G\MMod)$ consider the following
related argument.
Under the monoidal functor
$\Res^G_{1}:\Der(\kk G\MMod)\to \Der(\kk)$,
any tt-ring $A$ in~$\Der(\kk G\MMod)$ must go to an object
concentrated in degree 0, by the field case
(see Neeman~\cite[Prop.\,1.6]{Neeman15pp}). Hence, $A$
has only homology in degree zero and belongs to the image of the
fully faithful tensor functor $\kk G\MMod\hook \Der(\kk G\MMod)$.
We are therefore reduced to the module case and the same statement
holds for $\Der(\kk G\MMod)$ as for $\kk G\MMod$: Their
only commutative and separable rings are the announced~$\kk X$ for
finite $G$-sets~$X$.

The question for the stable category is much trickier, mostly
because the ``fiber" functor to the non-equivariant case,
$\Res^G_1:\kk G\sstab\to \kk \sstab=0$, is useless.

Our treatment starts with the case of $G=C_p$, cyclic of prime order.
This turns out to be the critical case. We then proceed relatively
easily to~$C_{p^n}$ by induction on~$n$. Only the case of $C_{4}$
requires an extra argument.

The reader might wonder how the result can be so difficult for
such a ``simple" category as $\kk C_{p^n}\sstab$. Let us point to the
fact that for the arguably even simpler, non-equivariant category
$\cat C=\kk\MMod$ of $\kk$-vector spaces, the proof requires a
couple of pages in DeMeyer-Ingraham~\cite[\S\,II.2]{DeMeyerIngraham71}.
The alternate proof of Neeman~\cite[\S\,1]{Neeman15pp} is equally
long. Our result relies on these predecessors. Most importantly,
the tensor product
in $\kk C_{p^n}\sstab$ becomes rather complicated, even for indecomposable
modules. See Formula~\eqref{eq:i@j} for $C_p$ itself. A critical new
ingredient in the stable category of~$C_p$ is the fact that the symmetric
module $S^{p-1}[i]$ over the indecomposable $\kk C_p$-module~$[i]$
of dimension~$i$ is projective, for every $i>1$. This fact was
established by Almkvist and Fossum in~\cite{AlmkvistFossum78}.
In addition, the Kelly radical of~$\kk C_p\sstab$ is a
tensor-ideal, a fact which we show in Section~\ref{se:Kelly}.
It is a very special feature of this case, as we also explain:
When $p^2$ divides the order of~$G$, the Kelly radical of
$\kk G\sstab$ is not a $\otimes$-ideal.
More generally, in Section~\ref{se:Kelly} we characterize completely
the smallest $\otimes$-ideal containing the Kelly radical,
for any finite group~$G$. This is Theorem~\ref{thm:tens-rad}
which is of independent interest.

The question discussed here is related to
the Galois group of the stable module
category as an $\infty$-category, as discussed by Mathew~\cite[\S\,9]{Mathew16},
although neither result seems to imply the other.

We remind the reader that for a finite group $G$ and a field $\kk$ of
characteristic $p > 0$, the stable category $\kk G\sstab$ is the category
whose objects are finitely generated $\kk G$-modules and whose morphism are
given by $\Hom_{\kk G\sstab}(M,N) = \Hom_{\kk G}(M, N)/ \Phom_{\kk G}(M,N)$
where $\Phom$ indicates those homomorphisms that factor through a projective
module.

\smallbreak
\textbf{Acknowledgements}:
Both authors would like to thank Bielefeld University for
its support and kind hospitality during a visit when
this work was undertaken. We are also thankful to
Danny Krashen, Akhil Mathew and
Greg Stevenson for valuable discussions.

\goodbreak
\section{Separable ring-objects}
\label{se:basics}%
\medbreak

In this section, we review the needed fundamental results on separable
ring-objects in tensor categories, not necessarily triangulated at first.

Assume that  $\cat C$ is a \emph{tensor category},
meaning an additive, symmetric
monoidal category such that $\otimes:\cat C\times\cat C\too \cat C$ is
additive in each variable. We denote by~$\unit$ the $\otimes$-unit.
A \emph{ring-object} $A$ in~$\cat C$ is a
triple $(A, \mu, u)$ where $A\in\Obj(\cat C)$, $\mu:A\otimes A\to A$
is an associative multiplication, $\mu(\mu\otimes A)=\mu(A\otimes \mu)$,
and the morphism~$u:\unit\to A$ is a two-sided unit,
$\mu(A\otimes u)=1_A=\mu(u\otimes A)$. (If $\cat C$ were not additive,
a common terminology would be ``monoid" instead of ``ring-object".)
The ring-object~$A$ is \emph{commutative} if $\mu(12)=\mu$, where
$(12):A\otimes A\isoto A\otimes A$ is the swap of factors.
By associativity, the composite of multiplications
$A\potimes{n}\otoo{\mu} A\potimes{n-1}\to \cdots \to
A\potimes{2}\otoo{\mu}A$ does not depend on the bracketing and
we simply denote it by~$\mu:A\potimes{n}\to A$.

In this setting, an \emph{$A$-module in the tensor category~$\cat{C}$}
is a pair $(M,\rho)$ where $M$ is an object of the given
category~$\cat C$ (not some `external' abelian group) and
$\rho:A\otimes M\to M$
is a morphism in~$\cat C$ satisfying the usual axioms of
associativity and unital action. Such modules and their
$A$-linear morphisms form an additive category~$A\MModcat{C}$.
It comes with the so-called \emph{Eilenberg-Moore} adjunction
\[
F_A:\cat C\adjto A\MModcat{C}:U_A
\]
where $F_A(X)=(A\otimes X,\mu\otimes X)$ is the free $A$-module
and its right adjoint $U_A(M,\rho)=M$ is the functor forgetting the action.
This material is classical, and is recalled with more
details in~\cite[\S\,2]{Balmer11} for instance.

\begin{Def}
\label{def:sep}%
A ring-object $A$ as above is \emph{separable} if there exists
$\sigma: A\to A\otimes A$ such that $\mu\sigma=1_A$ and
$\sigma\mu=(\mu\otimes A)(A\otimes \sigma)=(A\otimes \mu)(\sigma\otimes A)$.
This amounts to saying that $A$ is projective as an $A\otimes A\op$-module.
\end{Def}

\begin{Exa}
\label{exa:k(G/H)}%
As in the Introduction, for a subgroup~$H\le G$, the
separable commutative rings $A^G_H:=\kk(G/H)$ in~$\cat C=\kk G\sstab$
has multiplication $\mu:A^G_H\otimes A^G_H\to A^G_H$ extending
$\kk$-linearly the formulas $\mu(x\otimes x)=x$ and
$\mu(x\otimes x')=0$ for all $x\ne x'\in G/H$, and unit
$u:\kk\to A^G_H$ given by $u(1)=\sum_{x\in G/H}\,x$. The multiplication
$\mu$ is split by the map  $\sigma: A^G_H \to A^G_H\otimes A^G_H$,
that takes $x \in G/H$ to $\sigma(x) = x \otimes x.$
\end{Exa}

\begin{Prop}
\label{prop:no-nil}%
Let $A$ be a separable commutative ring-object in a
tensor category~$\cat C$. Then we have:
\begin{enumerate}[\rm(a)]
\item
\textbf{Relative semisimplicity of~$A$ over~$\cat C$}:
Let $f:M'\to M$ and $g:M\to M''$ be two morphisms of $A$-modules
in~$\cat C$ such that the underlying sequence of objects
$0\to M'\oto{f} M\oto{g} M''\to 0$ is split-exact in~$\cat C$.
Then the sequence is split-exact as a sequence of
$A$-modules, \ie $f$ admits an $A$-linear retraction~$r:M\to M'$ such that
$\smat{r\\g}: M\isoto M'\oplus M''$ is an isomorphism of $A$-modules.
\smallbreak
\item
\textbf{No nilpotence}: Suppose that $A=I\oplus J$ in~$\cat C$
and that $I$ is an ideal (\ie the morphism
$A\otimes I\to A\otimes A\oto{\mu} A$ factors via $I\hook A$).
Suppose that $I$ is nilpotent (\ie there exists $n\geq 1$
such that $I\potimes{n}\to A\potimes {n}\otoo{\mu}A$ is zero).
Then $I=0$.
\end{enumerate}
\end{Prop}

\begin{proof}
For~(a), consider a retraction $r:M\to M'$ of
$f:M'\to M$ in~$\cat C$, so that $r\,f=1_{M'}$.
Now, let $\bar r:M \to M'$ be the following composite:
\[
\xymatrix{
M \ar[r]^-{u\otimes 1}
& A\otimes M \ar[r]^-{\sigma\otimes 1}
& A \otimes A \otimes M \ar[r]^-{1\otimes \rho}
& A \otimes M \ar[r]^-{1\otimes r}
& A \otimes M' \ar[r]^-{\rho'}
& M'.
}
\]
This morphism is still a retraction of~$f$ but is now
$A$-linear. The reader unfamiliar with separability could
check these facts to appreciate the non-triviality of this
property. Indeed, the above construction $r\mapsto \bar r$
yields in general a well-defined map
$H:\Homcat{C}(M,M')\to \Hom_{A\MModcat{C}}(M,M')$ which
retracts the inclusion
$\Hom_{A\MModcat{C}}(M,M')\hook \Homcat{C}(M,M')$ and which
is natural in~$M$ and $M'$ in the sense that
$H(frf')=f\,H(r)f'$ whenever $f$ and $f'$ are $A$-linear.
See~\cite[2.9\,(1)]{BoehmBrzezinskiTomasz09}
or~\cite[Prop.\,6.3]{BruguieresVirelizier07} for details.

Part~(b) follows easily from~(a), since now we have that $A=I\oplus J$
as $A$-modules, that is, as ideals. Consider the unit morphism
$u:\unit\to A=I\oplus J$. The composition
\[
\unit=\unit\potimes{n}\otoo{u\potimes{n}}
A\potimes{n}=(I\oplus J)\potimes{n} \otoo{\mu} A
\]
is equal to~$u$ itself. Since $(I\oplus J)\potimes{n}=
I\potimes{n}\oplus (J\otimes(...))$, since $I$ is nilpotent
and since $J$ is an ideal, the above composition
factors via~$J\hook A$ for $n$ big enough. So the ideal $J\subseteq A$ contains the unit.
This readily implies $J=A$ and $I=0$ as claimed.
\end{proof}

\begin{Rem}
\label{rem:idemp}%
In general there are examples of separable commutative ring-objects in
the big stable category $\kk G\SStab$ for a finite group~$G$,
that differ from the objects associated to finite $G$-sets
as in the Introduction. These arise, for instance, as Rickard
idempotents~\cite{Rickard97}. Recall briefly, that to any
specialization-closed subset $Y$ in the
spectrum $V_G(\kk) = \text{Proj}(\HH^*(G,\kk))$
of homogeneous prime ideals in the cohomology ring of~$G$, we
associate an exact triangle in~$\kk G\SStab$
\[
\xymatrix{
\mathcal{E}_Y \ar[r]^\gamma & \kk \ar[r]^\lambda & \mathcal{F}_Y \ar[r] & {}
}
\]
where $\mathcal{E}_{Y}\otimes\mathcal{F}_{Y}=0$ and where $\mathcal{E}_{Y}$ belongs to
the localizing subcategory generated by~${\cat C}_Y:=\SET{M\in\kk G\sstab}{V_G(M) \subseteq Y}$
and $\mathcal{F}_Y$ to its orthogonal, \ie $\Hom_{\kk G\SStab}(M,\mathcal{F}_Y)=0$ for all~$M\in \cat C_Y$.
These properties uniquely characterize $\mathcal{E}_Y$ and~$\mathcal{F}_Y$.
Then there is a
multiplication $\mu:\mathcal{F}_Y \otimes \mathcal{F}_Y \to \mathcal{F}_Y$
inverse to the isomorphism
$\lambda \otimes \mathcal{F}_Y=\mathcal{F}_Y \otimes \lambda$,
turning $\mathcal{F}_Y$ into a tt-ring in~$\kk G\SStab$. The $\kk G$-module~$\mathcal{F}_Y$ is not finitely
generated as soon as $Y$ is non-empty and proper.

This phenomenon is a special case of the general observation that
a right Rickard idempotent in any tensor-triangulated category
is a tt-ring. Its category of modules is nothing but the
corresponding Bousfield (smashing) localization.
\end{Rem}

In the proof of our main theorem, we come across the
following tensor category. Let us describe its separable
commutative ring-objects.
\begin{Prop}
\label{prop:super}%
Let $\kk$ be a field of characteristic~2. The only commutative
separable ring in the $\otimes$-category of $\bbZ/2$-graded
$\kk$-vector spaces are concentrated in degree zero (\ie the separable $\kk$-algebras with trivial grading).
\end{Prop}

\begin{proof}
As extension-of-scalars from $\kk$ to any bigger field~$L/\kk$
is faithful, we can assume that~$\kk$ is separably closed.
The functor which maps a $\bbZ/2$-graded $\kk$-vector
space $(V_0,V_1)$ to the `underlying' $\kk$-vector space
$V_0\oplus V_1$ is a tensor functor. Suppose $A=(V_0,V_1)$
is a commutative separable $\bbZ/2$-graded $\kk$-algebra
and let us prove that $V_1=0$. Since $\kk$ is separably
closed, the underlying $\kk$-algebra of~$A$ is trivial.
Let $\varphi:\kk\times \cdots \times \kk\isoto V_0\oplus V_1$
be an isomorphism of ungraded $\kk$-algebras.
Consider $e=(0,\ldots,0,1,0,\ldots,0)\in\kk^{\times n}$
one of the idempotents, and let $\varphi(e)=v_0+v_1$
with $v_0\in V_0$ and $v_1\in V_1$ in~$A$. The
relation $\varphi(e)^2=\varphi(e)$, commutativity and
characteristic two, give $v_0=v_0^2+v_1^2$ and $v_1=2v_0v_1=0$.
So $v_0\in V_0$ is an idempotent. Hence, $V_0$ contains~$n$ orthogonal
idempotents, showing that $\dim_{\kk}V_0\ge n=\dim_{\kk}(A)$. Hence
$A=V_0$ and $V_1=0$ as claimed.
\end{proof}

\goodbreak
\section{The Kelly radical and the tensor}
\label{se:Kelly}%
\medbreak

Throughout this section,
$\kk$ is a field of positive characteristic~$p$ dividing the order of~$G$ and modules over a group algebra $\kk G$
are assumed to be finite dimensional. We
begin by recalling the definition of Kelly radical 
of a category~\cite{Kelly64}.
\begin{Def}
The \emph{radical} of an additive category
$\cat C$ is the ideal of morphisms
\[
\Radcat{C}(M,N)=\SET{f:M\to N}{\textrm{ for all }g:N\to M,\
1_M-gf\textrm{ is invertible}}.
\]
When $M=N$, the ideal $\Radcat{C}(M):=\Radcat{C}(M,M)$
is the Jacobson radical of the ring~$\End_{\cat C}(M)$.
\end{Def}

In this section, we give a characterization of the tensor-closure of the
Kelly radical, both in the module category $\kk G\mmod$ and in the stable category
$\kk G\sstab$. In particular, we show that if $G$ is a
cyclic $p$-group, then the Kelly radical is a tensor ideal. The results of
this section are far stronger than what is needed for later sections,
but they are of independent interest.

\begin{Rem}
\label{rem:ideal}%
Recall that in an additive category~$\cat C$ an
\emph{ideal of morphisms}~$\CI$ consists of a collection of subgroups
$\CI(M,N)\subseteq\Homcat{C}(M,N)$ for all object~$M,N$
(we only consider \emph{additive} ideals in this paper),
which is closed under composition:
\begin{equation}
\label{eq:ideal}
\Hom(N,N')\circ \CI(M,N)\circ \Hom(M',M)\subseteq\CI(M',N')\,.
\end{equation}
Then for any decompositions~$M\simeq M_1\oplus \cdots \oplus M_m$
and $N\simeq N_1\oplus \cdots \oplus N_n$ a morphism
$f\in\Homcat{C}(M,N)$ belongs to $\CI(M,N)$ if and only if
each \mbox{$f_{ji}=\pr_j\circ f\circ \inj_i$} belongs to
$\CI(M_i,N_j)$, where $\inj_i:M_i\into M$ and $\pr_j:N\onto N_j$
are the given injections and projections. Hence, an ideal~$\CI$
of morphisms in a Krull-Schmidt category~$\cat C$ is determined
by the subgroups $\CI(M,N)\subseteq\Homcat{C}(M,N)$ for
indecomposable~$M,N$. Conversely, a collection of such
subgroups~$\CI(M,N)\subseteq\Homcat{C}(M,N)$ for all
indecomposable~$M,N$ defines a unique
ideal~$\CI$ if~\eqref{eq:ideal} is satisfied for all
$M,M',N,N'$ indecomposable.

For any ideal~$\CI$, we can form the additive quotient category~$\cat C/\CI$
\begin{equation}
\label{eq:Q}%
Q:\cat C \onto \cat C/\CI
\end{equation}
which has the same objects as~$\cat C$ and morphisms
$\Homcat{C}(M,N)/\CI(M,N)$. When $\CI=\Radcat{C}$, we
have $1_M\notin\Radcat{C}(M)$ unless $M=0$. The corresponding
functor~$Q:\cat C\onto \cat C/\Radcat{C}$ is conservative
(detects isomorphisms).
\end{Rem}

\begin{Rem}
\label{rem:tens-id}%
When $\cat C$ is a tensor category, an ideal~$\CI$ of morphisms
is called a \emph{tensor ideal} (abbreviated $\otimes$-ideal)
if $f\otimes g\in\CI$ whenever $f\in\CI$. This is equivalent to
asking only $f\otimes L\in\CI(M\otimes L,N\otimes L)$ for
every $f\in\CI(M,N)$ and every object~$L$. In that case,
$\cat C/\CI$ becomes a $\otimes$-category and the quotient
$Q:\cat C\onto \cat C/\CI$ is a $\otimes$-functor.

It should be emphasized that the definition of the Kelly radical
$\Radcat{C}$ is not related to the
existence of a tensor structure
on~$\cat C$. In particular, for any specific $\otimes$-category~$\cat C$,
the ideal $\Radcat{C}$ may or may not be a $\otimes$-ideal.
So the quotient functor $Q:\cat C\to \cat C/\Radcat{C}$ is not
necessarily a $\otimes$-functor, even if in specific cases
$\cat C/\Radcat{C}$ admits some `natural' tensor structure for
independent reasons.
\end{Rem}

\begin{Def}
\label{def:Rat}%
We denote by $\Rat$ the smallest $\otimes$-ideal containing~$\Rad$, \ie
the $\otimes$-ideal it generates. We call $\Rat$ the
\emph{tensor-closure of the Kelly radical}.
\end{Def}

Our discussion of the tensor-closure $\Rat$ passes through the
algebraic closure~$\okk$ of~$\kk$. For this reason, we isolate
some well known facts as a preparation:
\begin{Prop}
\label{prop:okk}%
Let $\okk$ be an algebraic closure of~$\kk$. Let $M$ and~$N$ be
finite dimensional $\kk G$-modules, and consider the
$\okk G$-modules~$\okk \otimes_{\kk} M$ and~$\okk\otimes_{\kk} N$. Then:
\begin{enumerate}[\rm(a)]
\item
There is a canonical and natural isomorphism
\begin{equation}
\label{eq:okk}%
\Hom_{\okk G}(\okk \otimes_{\kk} M, \okk\otimes_{\kk} N) \simeq
\okk \otimes_{\kk} \Hom_{\kk G}(M, N)\,.
\end{equation}
\item
\label{it:Lam}%
Under~\eqref{eq:okk} for $M=N$, we have that
$\okk\otimes_{\kk}\Rad_{\kk G}(M)\subseteq\Rad_{\okk G}(\okk\otimes_{\kk} M)$.
\smallbreak
\item
\label{it:MN-summand}%
Suppose $M$ and $N$ are indecomposable.
Then $\okk \otimes_{\kk} M$ and $\okk \otimes_{\kk} N$ have a
nonzero direct summand in common if and only if $M \simeq N$.
\smallbreak
\item
\label{it:ksummand}%
Suppose that the trivial module $\okk$ is a
direct summand of the $\okk G$-module $\okk \otimes_{\kk} M$. Then
$\kk$ is a direct summand of~$M$.
\end{enumerate}
\end{Prop}

\begin{proof}
The canonical morphism $\okk \otimes \Hom_{\kk G}(M, N) \too \Hom_{\okk G}(\okk \otimes M, \okk\otimes N)$, between left exact functors in~$M$ (for~$N$ fixed) is an isomorphism when $M=\kk G$, hence also for every finitely presented $\kk G$-module~$M$. This gives~\eqref{eq:okk}. For~(b), it suffices to observe that $\okk\otimes_{\kk}\Rad_{\kk G}(M)$ is a nilpotent two-sided ideal of the ring $\okk \otimes_{\kk} \End_{\kk G}(M)$, which is isomorphic to the ring~$\End_{\okk G}(\okk \otimes_{\kk} M)$ by~(a). See~\cite[Thm.\,5.14]{Lam91} if necessary. For~(c), assume that $M \not\simeq N$ and suppose that $U$ is a
direct summand of both $\okk \otimes M$ and $\okk \otimes N$.
Then there exist homomorphisms $f: \okk \otimes M \to \okk \otimes N$ and
$g : \okk \otimes N \to \okk \otimes M$ such that $gf$ is an idempotent
endomorphism of~$\okk\otimes M$ with image isomorphic to~$U$.
By~\eqref{eq:okk}, $f=\sum_{i=1}^m a_i\otimes f_i$ and
$g=\sum_{j=1}^n b_j\otimes g_j$ for some $a_i,b_j\in \okk$ and
$f_i:M\to N$ and $g_j:N\to M$. As $M\not\simeq N$, all compositions
$g_jf_i:M\to M$ belong to the radical since they factor through~$N$.
Because $M$ is finite-dimensional, the radical of
$\Hom_{\kk G}(M,M)$ is nilpotent. So there exists an integer~$\ell$
such that $(gf)^\ell=0$. But $gf$ is idempotent, so~$gf=0$ and
therefore $U\simeq\textrm{im}(gf)=0$.
For~(d), we can assume~$M$ indecomposable. Then~(d) follows
from~(c) with~$N=\kk$.
\end{proof}

\begin{Rem}
Another tool in our discussion of~$\Rat$ is rigidity. Recall that a tensor category~$\cat C$ is \emph{rigid} if there exists a `dual' $(-)^\vee:\cat C\op\to \cat C$
such that every $M\in\cat C$ induces an adjunction
\[
\xymatrix@R=2em
{\cat C \ar@/_1em/[d]_-{M\otimes-}
\\
\cat C \ar@/_1em/[u]_-{M^\vee\otimes-}
}
\]
This holds for instance for $\cat C=\kk G\mmod$
or for $\cat C=\kk G\sstab$ with
$M^\vee=\Hom_{\kk}(M,\kk)$ with the usual $G$-module structure
$(g\cdot f)(m)=f(g\inv m)$.
The above adjunction comes with a unit $\eta_M:\unit\to M^\vee\otimes M$ and a counit $\eps_M:M\otimes M^\vee\to \unit$, which in our example are respectively the $\kk$-linear map $\kk \to M^\vee\otimes M\simeq \End_{\kk}(M)$ mapping $1\in\kk$ to the identity of~$M$, and the $\kk$-linear map given by the swap of factors followed by the trace: $M\otimes M^\vee\simeq M^\vee\otimes M\simeq \End_{\kk}(M)\otoo{\tr}\kk$.
\end{Rem}

Rigidity allows us to isolate a critical property of a module, which
is at the heart of the distinction between~$\Rad$ and~$\Rat$.

\begin{Def}
\label{def:xfaithful}%
A finitely generated $\kk G$-module~$M$
is said to be \emph{$\otimes$-faithful} provided the
functor $M\otimes-:\kk G\sstab\to \kk G\sstab$ is faithful.
\end{Def}

\begin{Rem}
We use the stable category, not the ordinary category, for the above
simple definition. In $\kk G\mmod$, every non-zero~$M$ induces a
faithful functor~$M\otimes-$. We are nevertheless going to give
several equivalent formulations in $\kk G\mmod$. Note that a projective
$\kk G$-module~$P$ is never $\otimes$-faithful, since $P=0$ in~$\kk G\sstab$.
\end{Rem}

\begin{Exa}
\label{exa:pdiv}%
A $\kk G$-module~$M$ such
that $\dim(M)$ is prime to~$p=\chara(\kk)$ is $\otimes$-faithful since
$\eta_M:\unit\to M^\vee\otimes M$ is split by
$\dim(M)\inv\cdot\tr:M^\vee\otimes M\to \kk$.
A converse holds when $\kk$ is algebraically closed,
as we recall in Theorem~\ref{thm:tracemap} below.
\end{Exa}

\begin{Prop}
\label{prop:xfaithful}%
Let $M$ be a finitely generated $\kk G$-module. The following properties are equivalent:
\begin{enumerate}[\rm(i)]
\smallbreak
\item
\label{it:x-1}%
The $\kk G$-module $M$ is $\otimes$-faithful (Definition~\ref{def:xfaithful}).
\smallbreak
\item
\label{it:x-1'}%
Some indecomposable summand of~$M$ is $\otimes$-faithful.
\smallbreak
\item
\label{it:x-X}%
There exists a finitely generated $\kk G$-module~$X$ such that
$X\otimes M$ is~$\otimes$-faithful.
\smallbreak
\item
\label{it:x-3}%
The unit $\eta:\kk \to M^\vee \otimes M$ is a split monomorphism of
$\kk G$-modules.
\smallbreak
\item
\label{it:x-4}%
The unit $\eta:\kk \to M^\vee \otimes M$ is a split monomorphism
in the stable category~$\kk G\sstab$.
\smallbreak
\item
\label{it:x-3'4'}%
The trace $\tr:M^\vee \otimes M\to \kk$ (or equivalently the
counit $\eps_M:M\otimes M^\vee\to \kk$) is a split epimorphism
of $\kk G$-modules, or equivalently in the stable category.
\smallbreak
\item
\label{it:pdiv2}%
$\kk$ is a direct summand of $M^\vee \otimes M$ in~$\kk G\mmod$,
or equivalently in~$\kk G\sstab$.
\smallbreak
\item
\label{it:x-X2}%
$\kk$ is a direct summand of $X\otimes M$ for some finitely
generated $\kk G$-module~$X$.
\end{enumerate}
\medbreak
\noindent
If $\okk$ is an algebraic closure of~$\kk$, then the above are further
equivalent to:
\begin{enumerate}[\rm(i)]
\setcounter{enumi}{8}
\smallbreak
\item
\label{it:x-okk}%
The $\okk G$-module $\okk\otimes_{\kk} M$ is $\otimes$-faithful.
\smallbreak
\item
\label{it:pdiv3}%
Some direct summand of $\okk \otimes_\kk M$
has dimension that is not divisible by~$p$.
\end{enumerate}
\end{Prop}

\begin{proof}
Recall that $\kk G\mmod$ and $\kk G\sstab$ are Krull-Schmidt categories
and that $M$ has the same indecomposable summands in both, except for
the projectives which vanish stably. Hence, the two formulations
of~\eqref{it:pdiv2} are indeed equivalent (and~\eqref{it:x-X2} is unambiguous).
It is straightforward to check
\eqref{it:x-1}$\iff$\eqref{it:x-1'}$\iff$\eqref{it:x-X} from
Definition~\ref{def:xfaithful}. Also obvious are
\eqref{it:x-3}\then\eqref{it:x-4}\then\eqref{it:pdiv2}\then\eqref{it:x-X2}\then\eqref{it:x-1}.

Let us show that \eqref{it:x-1}\then\eqref{it:x-3}.
As $\cat C=\kk G\sstab$ is a rigid tensor-\emph{triangulated} category, we
can choose an exact triangle $N\oto{\xi}\unit\oto{\eta}M^\vee\otimes M\to\Sigma N$ in~$\cat C$.
The unit-counit relation
shows that $M\otimes\eta$ is a split monomorphism. Hence,
$M\otimes \xi=0$ in~$\cat C$, and $\xi=0$ since
$M\otimes-:\cat C\to \cat C$ is assumed faithful.
Consequently,  $\eta$ is a split monomorphism, by a standard property of
triangulated categories, see~\cite[Cor.\,1.2.7]{Neeman01}.

Similarly, \eqref{it:x-3'4'}\then\eqref{it:x-1} is trivial and \eqref{it:x-1}\then\eqref{it:x-3'4'} is proven as above.

At this stage, we know that \eqref{it:x-1}--\eqref{it:x-X2} are all equivalent. Now, property~\eqref{it:pdiv2} holds for~$M$ over~$\kk$ if and if it holds for~$\okk\otimes_\kk M$ over~$\okk$ by Proposition~\ref{prop:okk}\,\eqref{it:ksummand}. Hence, \eqref{it:x-1}--\eqref{it:x-X2} are also equivalent to~\eqref{it:x-okk}. We already saw that~\eqref{it:pdiv3}\then\eqref{it:x-okk} in Example~\ref{exa:pdiv}. The converse, \eqref{it:x-okk}\then\eqref{it:pdiv3}, holds by the following more general theorem of
Dave Benson and the second author (applied in the case $\kk=\okk$).
\end{proof}

\begin{Thm}[{\cite{BensonCarlson86}}] \label{thm:tracemap}
Suppose that $M$ and $N$ are absolutely
indecomposable $\kk G$-modules (\ie remain indecomposable
over the algebraic closure)
and suppose that the trivial module $\kk$
is a direct summand of $M \otimes N$.
Then $\dim(M)$ is not divisible by~$p$, $N \simeq M^\vee$, the
multiplicity
of $\kk$ as direct summand of $M \otimes N$ is one, the unit
$\eta_M:\kk\to M^\vee\otimes M$ (mapping 1 to $1_M$) is a split
monomorphism and the trace
$\tr:M^\vee \otimes M \to \kk$ is a split epimorphism.
\end{Thm}

\begin{Rem}\label{rem:notalgclosed}
The reason for the assumption of absolute indecomposability is
illustrated in an easy example. Let $G = \langle x \rangle \simeq C_3$, be
a cyclic group of order~3, and $\kk = \mathbb{F}_2$, the prime field
with two elements. The algebra $\kk G$ is a semisimple and as a module
over itself it decomposes $\kk G \simeq \kk \oplus M$, where
$M$ has dimension~2.
If a cube root of unity $\zeta$ is adjoined to $\kk$, then $M$ splits as a
sum of two one-dimensional modules on which $x$ acts by multiplication
by $\zeta$ on one and by $\zeta^2$ on the other.
Then it is not difficult to see
that $M \otimes M \simeq \kk \oplus \kk \oplus M$, since $x$ has an
eigenspace, with eigenvalue one, of dimension~2 on the tensor product.
Of course, in this example, the characteristic of the field does not divide
the order of the group. Another example can be constructed by inflating
this module to $\kk G$ where $G$ is the alternating group $A_4$, along the
map $G \to C_3$ whose kernel is the Sylow 2-subgroup. More complicated
examples also exist.
\end{Rem}

The next lemma is a corollary of the multiplicity-one property in Theorem~\ref{thm:tracemap}.
\begin{Lem}
\label{lem:traces}%
Assume that $\kk = \okk$ is algebraically closed.
Let $M$ be an indecomposable $\kk G$-module of dimension prime to~$p$.
Let $j:\kk \into M^\vee\otimes M$ be any split monomorphism and
$q:M^\vee\otimes M\onto \kk$ any split epimorphism. Then:
\begin{enumerate}[\rm(a)]
\item
The composite $q\circ j:\kk \to \kk$ is not the zero map.
\smallbreak
\item
Let $f:M\to M$ be any morphism. Then the composite
$q(1\otimes f)j:\kk \to M^\vee\otimes M\otoo{1\otimes f}
M^\vee\otimes M\to \kk$ is a non-zero multiple of the trace of~$f$.
\end{enumerate}
\end{Lem}

\begin{proof}
  Since $M^\vee\otimes M\simeq \kk\oplus L$ where $L$ contains no $\kk$
  summand, the split morphisms~$j$ and $q$ must be respectively a
  non-zero multiple of the (canonical) split morphisms
  $\eta:\kk \into M^\vee\otimes M$ and $\tr:M^\vee\otimes M\onto \kk$,
  plus morphisms factoring through~$L$, which are in particular
  in the Kelly radical. Computing the composite in~(b), using
  that $\Rad(\kk)=0$, we see that $q(1\otimes f)j$ is a
  non-zero multiple of
\[
\kk \otoo{\eta} M^\vee\otimes M \otoo{1\otimes f}M^\vee\otimes M\otoo{\tr}\kk
\]
which is the trace of~$f$. Part~(a) follows from~(b) for $f=1_M$.
\end{proof}

Let us return to our discussion of the tensor-closure~$\Rat$ of the radical.
The relevance of $\otimes$-faithfulness (Definition~\ref{def:xfaithful}) for this question
is isolated in the following result. Recall that $p=\chara(\kk)$ divides~$|G|$.

\begin{Prop}
\label{prop:non-ff}%
Let $M$ be a finitely generated $\kk G$-module which is \emph{not} $\otimes$-faithful.
Then the identity $1_M$ of~$M$ belongs to the tensor-closure $\Ratcat{C}$ of
the radical, both in~$\cat C=\kk G\mmod$ and in~$\cat C=\kk G\sstab$. Hence, $\Ratcat{C}(M,M)=\Homcat{C}(M,M)$.

In particular, the Kelly radical of $\cat C=\kk G\mmod$
is never a $\otimes$-ideal. Moreover,
if there exists such an~$M$ not projective, then
the Kelly radical is not $\otimes$-ideal in the
stable category~$\cat C=\kk G\sstab$.
\end{Prop}

\begin{proof}
By assumption, the unit $\eta_M:\unit\to M^\vee\otimes M$ is not
a split monomorphism. This means that $\eta_M$ belongs to the
radical, as $\Endcat{C}(\unit)\simeq \kk$. Hence, the morphism
$M\otimes \eta_M:M\to M\otimes M^\vee\otimes M$ belongs
to~$\Ratcat{C}$. By the unit-counit relation,
we have that $1_M=(\eps_M\otimes M)\circ(M\otimes \eta_M)$.
Hence, $1_M$ belongs to the ideal of morphisms~$\Ratcat{C}$
as claimed. This phenomenon readily implies that the Kelly
radical is not a $\otimes$-ideal if $M\neq 0$ in the category~$\cat C$,
since the identity of a non-zero object never belongs to~$\Radcat{C}$.
In the ordinary category~$\cat C=\kk G\mmod$, the free module $M=\kk G$
gives an example of such a non-zero~$M$.
On the other hand, we need $M$ to be not projective in order
to have that $M\neq 0$ in the stable category~$\cat C=\kk G\sstab$.
\end{proof}

We are therefore naturally led to consider the following ideal of morphisms,
first in~$\kk G\mmod$ and later in~$\kk G\sstab$ (Definition~\ref{def:rat-stab}).

\begin{Def}
\label{def:rat-mod}%
For $M$ and $N$ indecomposable $\kk G$-modules, let $\CI(M,N)$
be the subspace
of $\Hom_{kG}(M, N)$ defined as follows.
\begin{enumerate}[\rm(1)]
\item
\label{it:CI-1}%
If $M \not\simeq N$, then $\CI(M, N) := \Rad(\Hom_{kG}(M,N)) = \Hom_{kG}(M,N)$.
\item
\label{it:CI-2}%
If $M$ is not $\otimes$-faithful, then $\CI(M, M) := \Hom_{kG}(M,M)$.
\item
\label{it:CI-3}%
If $M$ is  $\otimes$-faithful,
then $\CI(M, M) := \Rad(\Hom_{kG}(M,M))$.
\end{enumerate}
(When $N\simeq M$, we define $\CI(M,N)$ via~\eqref{it:CI-2}
or~\eqref{it:CI-3} transported by any such isomorphism.)
This collection of $\CI(M,N)$ is closed under composition,
as discussed in Remark~\ref{rem:ideal}. Hence, it defines a unique
ideal of morphisms in~$\kk G\mmod$, still denoted~$\CI$.
It clearly contains the radical, from which it only differs
in case~\eqref{it:CI-2}.
\end{Def}

As earlier, let us see that this ideal is stable under algebraic field extensions.
\begin{Lem} \label{lem:fieldextend}
Let $\CI_{\kk}$ denote the ideal defined in Definition~\ref{def:rat-mod}
for $\kk G$-modules. Let $\okk$ be an algebraic closure of~$\kk$.
Suppose that $M$ and $N$ are $\kk G$-modules. A map $f: M \to N$
belongs to $\CI_{\kk}(M,N)$ if and only if
$\okk \otimes f$ belongs to $\CI_{\okk}(\okk \otimes M, \okk \otimes N)$.
\end{Lem}

\begin{proof}
We may assume that $M$ and $N$ are indecomposable.
If $M\not\simeq N$, then no nonzero summand of~$\okk \otimes M$
is isomorphic to any summand of~$\okk \otimes N$ by
Proposition~\ref{prop:okk}\,\eqref{it:MN-summand}. Then
$\CI_{\kk}(M,N)=\Hom_{\kk G}(M,N)$ and similarly for~$\okk$,
so there is nothing to prove about~$f$.
Suppose therefore that $M\simeq N$. Recall from
Proposition~\ref{prop:xfaithful} that $M$ is $\otimes$-faithful
if and only if $\okk\otimes M$ is. So, if $M$ is
not $\otimes$-faithful, then neither is any summand
of~$\okk\otimes M$ and again $\CI_{\kk}(M,M)$ and
$\CI_{\okk}(\okk\otimes M,\okk\otimes M)$ are the entire groups
of homomorphisms and there is nothing to prove about~$f$.

Let us then assume $M$ $\otimes$-faithful and take $f:M\to M$.
If $f$ does not belong to $\CI_{\kk}(M,M)=\Rad_{\kk G}(M)$ then $f$
is an isomorphism and then so is $\okk \otimes f$. As one summand
of~$\okk\otimes M$ is $\otimes$-faithful, the isomorphism
$\okk\otimes f$ does not belong
to~$\CI_{\okk}(\okk\otimes M,\okk\otimes M)$.
Conversely, suppose that $f:M\to M$ belongs to~$\CI_{\kk}(M,M)$,
which is the radical of $\End_{\kk G}(M)$.
By Proposition~\ref{prop:okk}\,\eqref{it:Lam},
$\okk \otimes f$ belongs to the radical,
which is contained in the larger ideal~$\CI_{\okk}$.
\end{proof}

\begin{Thm}
\label{thm:tens-rad}%
The ideal $\CI$ of Definition~\ref{def:rat-mod} is the tensor-closure
$\Rat$ of the Kelly radical (Definition~\ref{def:Rat})
of the category $\kk G\mmod$.
\end{Thm}

The critical point is the following:
\begin{Lem}
\label{lem:tens-rad}%
Assume that $\kk = \okk$ is algebraically closed.
Let $M$ and $N$ be indecomposable $\kk G$-modules of dimension prime
to~$p$ and let $f:M\to N$ be a homomorphism. Suppose that
$X$ is a $\kk G$-module
such that there is a common indecomposable summand $U\le M\otimes X$ and
$U\le N\otimes X$ of dimension prime to~$p$, with split injection
$i:U\into M\otimes X$ and split projection $p:N\otimes X\onto U$.
Suppose that $p\circ (f\otimes X) \circ i:U\to U$ is an isomorphism.
Then $f:M\to N$ is an isomorphism.
\end{Lem}

\begin{proof}
Let $g:=p\circ (f\otimes 1_X) \circ i:U\isoto U$ be our isomorphism.
Tensoring with $U^\vee$, we have an automorphism $g\otimes 1$
of~$U\otimes U^\vee$. The composite
\[
\xymatrix{
\kk \ar@{ >->}[r]^-{\eta_U}
& U \otimes U^\vee \ar[r]^-{i \otimes 1} \ar@/_2em/[rrr]^-{g\otimes 1}_-{\simeq}
& M \otimes X \otimes U^\vee \ar[r]^-{f \otimes 1 \otimes 1}
& N \otimes X \otimes U^\vee \ar[r]^-{p\otimes 1}
& U \otimes U^\vee \ar@{->>}[r]^-{\tr_U}
& \kk
}
\]
is non-zero by Lemma~\ref{lem:traces}\,(a), that is, an isomorphism.
Decomposing $X\otimes U^\vee$ into a sum of indecomposable summands~$V$,
the above isomorphism $\kk\to \kk$ is the sum of the corresponding
compositions
\[
\xymatrix{
\kk \ar@{->}[r]^-{}
& M \otimes V \ar[r]^-{f \otimes 1}
& N \otimes V \ar@{->}[r]^-{}
& \kk
}
\]
over all these~$V$. This holds because the middle map
$f\otimes 1\otimes 1$ above ``is" the identity on the
$X\otimes U^\vee$ factor. Since the sum of these morphisms is non-zero,
one of them must be non-zero, \ie an isomorphism, for some~$V$.
Applying Theorem~\ref{thm:tracemap} to $M\otimes V$ and again
to $N\otimes V$, we have that $V\simeq M^\vee$ and that
$V\simeq N^\vee$. In particular, $M\simeq N$. Replacing $N$ by $M$
using such an isomorphism, we can assume that $f:M\to M$ is
an endomorphism. The above isomorphism $\kk \to \kk$ now becomes
\[
\xymatrix{
\kk \ar@{ >->}[r]^-{j}
& M \otimes M^\vee \ar[r]^-{f \otimes 1}
& M \otimes M^\vee \ar@{->>}[r]^-{q}
& \kk
}
\]
for some morphisms~$j$ and~$q$ which must be a split mono and a
split epi respectively, since that composite is an isomorphism.
By Lemma~\ref{lem:traces}\,(b) this composite is also a
non-zero multiple of the trace of~$f$. Because the composite is an isomorphism,
$\tr(f)\neq 0$, and therefore $f$ cannot be
nilpotent. It follows that $f$ cannot belong to the radical of the
finite-dimensional $\kk$-algebra~$\End_{\kk G}(M)$. Hence, $f$ is invertible,
as claimed.
\end{proof}

\begin{proof}[Proof of Theorem~\ref{thm:tens-rad}]
We already know that the Kelly radical is contained in $\CI$.
The first thing to note
is that $\CI$ is contained in the tensor ideal $\Rat$
generated by the radical. This follows from the definition of~$\CI$ and Proposition~\ref{prop:non-ff}.

It remains to show that $\CI$ is a tensor ideal.
To begin, assume that $\kk = \okk$ is algebraically closed.
In this case an indecomposable $\kk G$-module is $\otimes$-faithful
if and only if its dimension is not divisible by~$p$.
Let $f\in\CI(M,N)$ with $M$ and $N$ indecomposable,
and let $X$ be an object.  We want
to show that $f\otimes X$ belongs to~$\CI$. To test this, we need
to decompose $M\otimes X$ and $N\otimes X$ into a sum of indecomposable.
Suppose ab absurdo that $f\otimes X$ does not belong
to~$\CI(M\otimes X,N\otimes X)$. Since $\CI(-,-)$ is often equal
to the whole of $\Hom(-,-)$, the only way that $f\otimes X$
cannot belong to~$\CI$ is that $M\otimes X$ and $N\otimes X$
admit a common direct summand $U$, of dimension prime to~$p$,
on which $f\otimes X$ is invertible (see case~\eqref{it:CI-3}
of Definition~\ref{def:rat-mod}). So~$U$ is $\otimes$-faithful,
hence so are $M\otimes X$ and $N\otimes X$, and therefore $M$ and~$N$ as well;
see Proposition~\ref{prop:xfaithful}, \eqref{it:x-1'}\then\eqref{it:x-1} and \eqref{it:x-3}\then\eqref{it:x-1}.
Therefore $\CI(M,N)=\Rad(M,N)$. In summary, $f\in \CI(M,N)$ is non-invertible but
$f\otimes X$ is invertible on some common indecomposable
summand~$U$ of $M\otimes X$ and $N\otimes X$, of dimension prime to~$p$.
This is exactly the situation excluded by Lemma~\ref{lem:tens-rad}
(since $\kk$ is algebraically closed).

Now consider the case of a general field $\kk$, perhaps not algebraically
closed. Suppose that $f: M \to N$ is in $\CI_{\kk}(M,N)$. Let $X$ be
any $\kk G$-module. Then $\okk \otimes (f \otimes X)
= (\okk \otimes f) \otimes (\okk \otimes X)$ is in
$\CI_{\okk}(\okk \otimes (X \otimes M), \okk \otimes (X \otimes N))$.
But now Lemma~\ref{lem:fieldextend}, implies that
$f \otimes X$ is in $\CI_{\kk}(M\otimes X, N \otimes X)$. Thus
$\CI_{\kk}$ is a tensor ideal and the proof is complete.
\end{proof}

\begin{Rem}
\label{rem:rat-stab}%
Everything that we have done in the module category will translate directly
to the stable category, except that the ideal needs to be defined somewhat differently.
Recall that $\Phom_{\kk G}(M,N)$ is the subspace of
$\Hom_{\kk G}(M,N)$ consisting
of all homomorphisms from $M$ to $N$ that factor through a projective
module. It is very easy to see that $\Phom_{\kk G}(M,N) \subseteq \CI(M,N)$.
Indeed, the only case, for $M$ and $N$ indecomposable, that $\Phom_{\kk G}(M,N)$
is not in the Kelly radical, occurs when $M \simeq N$ is projective, and
no projective module is $\otimes$-faithful.
\end{Rem}

\begin{Def}
\label{def:rat-stab}%
In the stable category, we define
\[
\CI_s(M, N) \ = \ \CI(M,N)/ \Phom_{\kk G}(M,N)\,.
\]
This clearly is an ideal in the stable category $\kk G\sstab$.
Explicitly, from Definition~\ref{def:rat-mod}, we have for $M$ and $N$ indecomposable in~$\cat C:=\kk G\sstab$:
\begin{enumerate}[\rm(1)]
\item
\label{it:CJ-1}%
If $M \not\simeq N$, then $\CI_s(M, N) := \Radcat{C}(M,N) = \Homcat{C}(M,N)$.
\item
\label{it:CJ-2}%
If $M$ is not $\otimes$-faithful, then $\CI_s(M, M) := \Homcat{C}(M,M)$.
\item
\label{it:CJ-3}%
If $M$ is  $\otimes$-faithful,
then $\CI_s(M, M) := \Radcat{C}(M,M)$.
\end{enumerate}
\end{Def}

\begin{Thm}
\label{thm:tens-rad2}%
The ideal $\CI_s$ is the tensor-closure $\Rat$
of the Kelly radical (Definition~\ref{def:Rat})
of the category $\kk G\sstab$.
\end{Thm}

\begin{proof}
By Theorem~\ref{thm:tens-rad}, $\CI$ is a $\otimes$-ideal,
and since $\Phom_{\kk G}$ is also a $\otimes$-ideal,
so is~$\CI_s=\CI/\Phom_{\kk G}$. The rest follows easily
as before. Indeed, $\CI_s$ clearly contains the radical,
and agrees with it in most cases, except in the case of~$\CI_s(M,M)$
for $M$ not $\otimes$-faithful where $\CI_s(M,M)=\Homcat{C}(M,M)$.
But in that case, this is also $\Ratcat{C}(M,M)$ by
Proposition~\ref{prop:non-ff} for $\cat C=\kk G\sstab$.
\end{proof}

This leads directly to the following.
\begin{Thm}\label{thm:rad}%
Let $\kk$ be a field of characteristic~$p>0$ and let $G$ be a finite group of order divisible by~$p$.
Then the Kelly radical $\Radcat{C}$ of the category~$\cat C=\kk G\sstab$
is a $\otimes$-ideal if and only if $p^2$ does not divide the order of~$G$.
\end{Thm}

\begin{proof}
First suppose that $p^2$ divides the order of $G$. Let $Q$ be a subgroup
of order $p$ in $G$ and let $M = \kk_Q^{\uparrow G} =
\kk G \otimes_{\kk Q} \kk_Q$ where $\kk_Q$ is the trivial $\kk Q$-module.
Let $S$ be a Sylow $p$-subgroup of $G$ that contains $Q$.
By the Mackey Theorem, we have that
\[
M_{\downarrow S} \quad  \simeq \quad \bigoplus_{SxQ}
\kk_{S \cap xQx^{-1}}^{\uparrow S}
\]
where the sum is over a collection of representatives of the
$S$-$Q$ double cosets in~$G$.
The modules $\kk_{S \cap xQx^{-1}}^{\uparrow S}$ are absolutely indecomposable, have
dimension divisible by $p$ and are not projective if $S \cap xQx\inv \neq \{1\}$.
Hence, some non-projective summand of $M$ must fail to be $\otimes$-faithful,
and the Kelly radical is not tensor closed by Proposition~\ref{prop:non-ff}.

On the other hand, suppose that a Sylow $p$-subgroup $S$ of $G$ is cyclic
of order~$p$. We show that every non-projective
indecomposable $\okk G$-module has
dimension prime to~$p$. This implies that every $\kk G$-module is
$\otimes$-faithful by Proposition~\ref{prop:xfaithful}
\eqref{it:pdiv3}\then\eqref{it:x-1}. So assume that $\kk = \okk$.
Let $S = \langle h \rangle$, and $t = h-1$, so that $\kk S = \kk[t]/(t^p)$.

First consider the case that $S$ is normal in $G$. Let $M$ be an
indecomposable $\kk G$-modules. Then it is known that $M$ is uniserial,
meaning that the subsets $M_i = t^iM$ are
$\kk G$-submodule for $i = 0, \dots, p-1$, and the
quotients $M_i/M_{i+1}$ are all irreducible and conjugate to
one another. Moreover, because $M$ is not projective, $M_{p-1} = \{0\}$.
Thus the dimension of $M$ is $r\cdot\dim(M/M_1)$ where $r$ is the least integer
such that $M_r = \{0\}$.  The quotient $M/M_1$ is an irreducible
$\kk G/S$-module. Because $\kk$ is algebraically closed, the dimension of
$M/M_1$ divides $\vert G/S \vert$ and is prime to~$p$
(see \cite{CurtisReiner} (33.7) which applies also in this case).
Hence $M$ has dimension prime to~$p$.

If $S$ is not normal in $G$, then let $N = N_G(S)$.
Let $M$ be a non-projective
indecomposable $\kk G$-module. Then $M$ is a direct summand of
$U^{\uparrow G}$ for $U$ an indecomposable  $\kk N$-module that
is a direct summand of
the restriction $M_{\downarrow N}$. Note that $U$ is not projective
as otherwise $M$ is also projective. Thus, by the previous case,
the dimension of $U$  is not divisible by $p$.  By the Mackey Theorem
\[
(U^{\uparrow G})_{\downarrow S} \ \simeq \
\bigoplus_{SxN} ((x \otimes U)_{\downarrow S \cap xNx\inv})^{\uparrow S}
\]
where the sum is over a set of representatives of the $S$-$N$ double cosets in~$G$.
But notice that $S \cap xNx\inv = \{1\}$ if $x \not\in N$. Hence,
$U^{\uparrow G}$ can have only one non-projective direct summand which must
be~$M$.
All other direct summands must have dimension divisible by~$p$.
Because $\dim(U^{\uparrow G}) = \vert G:N \vert \  \dim(U),$ we have that $p$
does not divide the dimension of~$M$.
\end{proof}

\begin{Exa}
\label{exa:C_p^n}%
Let $\cat C=\kk C_{p^n}\sstab$ the stable module category over a
cyclic $p$-group~$C_{p^n}=\langle \,g\Mid g^{p^n}=1\,\rangle$.
For $\kk$ of characteristic~$p$, we have a ring isomorphism
$\kk C_{p^n}\isoto \kk [t]/t^{p^n}$ given by $g\mapsto t+1$.
Every indecomposable module has the form
\[
[i]:=\kk [t]/t^i
\]
for $i=1,\ldots, p^{n}$, the last one being projective. Hence,
the indecomposable objects in the stable category~$\cat C$ are
the $[i]$ for $1\le i\le p^{n}-1$. The Kelly radical of~$\cat C$
is generated by the morphisms $\alpha_i:[i]\to [i+1]$ given by
multiplication by~$t$ and the morphisms $\beta_i:[i]\to [i-1]$
given by the projection. In particular, the radical of
$\Endcat{C}([i])$ is generated by the morphism
$\beta_{i+1}\alpha_i:[i]\to [i]$ given by multiplication by~$t$.
The modules are all absolutely indecomposable so that none of
this depends heavily on the field~$\kk$,
as long as it has characteristic~$p$, of course.
Consequently, the Kelly radical is preserved under field
extensions $\kk C_{p^n}\sstab\too \kk' C_{p^n}\sstab$.

Consider the quotient
$\cat C\ \overset{Q}\onto\ \cat D:=\cat C/\Radcat{C}$ of~\eqref{eq:Q}.
In this example, the category~$\cat D$
consists simply of $p^n-1$ copies of the category of $\kk$-vector
spaces, since we have $\End_{\cat D}(Q([i]))\simeq \kk$ for all~$i$
and $\Hom_{\cat D}(Q([i]),Q([j]))=0$ for $i\neq j$. There is a
natural component-wise tensor on~$\cat D$ in this example.
However, this tensor on~$\cat D$ never makes the quotient
functor $Q:\cat C\to \cat D$ into a $\otimes$-functor, except
for $G=C_2$ where $Q$ is an isomorphism. For $n=1$, we
have seen that $\Radcat{C}$ is a $\otimes$-ideal, hence there
is \emph{another} tensor structure on~$\cat D$ which makes $Q$
a $\otimes$-functor. For $n\ge 2$, the radical is simply not
a $\otimes$-ideal. (See Theorem~\ref{thm:rad}.)

In the case of $G = C_{p}$, the fact that the Kelly radical of the
stable category is a tensor ideal can also be deduced from the
tensor formula, for $i\le j$:
\begin{equation}
\label{eq:i@j}%
[i]\otimes [j]\simeq\left
\{\begin{array}{cl}{} [j-i+1]\oplus [j-i+3] \oplus \cdots\oplus [j+i-1]
\kern2em & \textrm{if } i+j\le p\\{} [j-i+1]\oplus [j-i+3]
\oplus \cdots\oplus [2p-i-j-1] & \textrm{if } i+j > p.
\end{array}\right.
\end{equation}
This formula is a consequence of a calculation of
Premet~\cite{Premet91}.
See~\cite[Cor.\,10.3]{CarlsonFriedlanderPevtsova08} for details.
Observe that all indecomposable summands~$[k]$ of
$[i]\otimes [j]$ have the same parity as $j-i+1$ (even for $i\ge j$ of course since $\otimes$ is symmetric).
In particular, every morphism $f$ between $[i]\otimes [j]$
and $[i\pm1]\otimes [j]$ belongs to the radical since no
summand of the source of~$f$ is isomorphic to any summand of
its target. On the other hand, we saw that the radical is
generated by the morphisms $\alpha_i:[i]\to [i+1]$
(multiplication by~$t$) and $\beta_i:[i]\to [i-1]$ (projection).
It follows that $\alpha_i\otimes [j]$ and $\beta_i\otimes[j]$
belong to the radical for all~$j$.
\end{Exa}

\goodbreak
\section{The case of the group of prime order}
\label{se:C_p}%
\medbreak

Let $p$ be a prime, $C_p$ the cyclic group of order~$p$ and $\kk$
a field of characteristic~$p$.

\begin{Thm}
\label{thm:C_p}%
Let $A\in\kk C_p\SStab$ be a tt-ring in the (big) stable module
category. Then there exists finitely many finite separable
field extensions $L_1,\ldots, L_n$ over~$\kk$ such that
$A\simeq L_1\times\cdots\times L_n$ as tt-rings in~$\kk C_p\SStab$,
where $L_1\times\cdots\times L_n$ is equipped with trivial $C_p$-action.
\end{Thm}

We need the following general preparations.
\begin{Rem}
\label{rem:M^G-M_G}%
Let $S$ be a finite group whose order is invertible in~$\kk$. Let $M$ be
a finite dimensional $\kk S$-module. Suppose that $M^S=0$, meaning $M$
has no non-trivial $S$-fixed vector. Then there is also no
nonzero $kS$-homomophism
$M \to \kk$, since such a map
would  split by semisimplicity of~$\kk S$, and thus $\kk$ would be a
direct summand of~$M$. It follows that
any $\kk$-linear map $\nu:M\to M'$ such that $\nu(s m)=\nu(m)$
for all~$m\in M$
and all $s\in S$ must be zero, since such a map $\nu$ has to
factor through a trivial $kS$-module.

In the proof of Theorem~\ref{thm:C_p}, we use the above argument in a slightly
more general setting, where $M$ is an object of a category
$\cat D=\oplus_{i=1}^m (\kk\MMod)_i$ obtained by
taking a finite (co)product of copies of the category~$\kk\MMod$
(as additive categories). Since
two copies of~$\kk\MMod$ for different indices have no non-zero
morphisms between them in~$\cat D$,
one easily reduces to the above case.
\end{Rem}

\begin{Rem}
\label{rem:fixed}%
More generally, let $\cat C$ be a $\kk$-linear idempotent-complete
category, and let $S$ be a finite group whose order is invertible
in~$\kk$. Let $M$ be an object of~$\cat C$ on which $S$ acts,
in the sense that we have a group homomorphism $S\to \Aut_{\cat C}(M)$.
We can then describe the $S$-fixed subobject $M^S$ as an explicit
direct summand of~$M$, namely the summand corresponding to the idempotent
endomorphism given by the image of the central idempotent
$e=\frac{1}{|S|}\sum_{s\in S}s$ of~$\kk S$ in~$\End_{\cat C}(M)$.
So we have $M^S=e\cdot M=\ker(1-e)$ and $M=e\cdot M\oplus (1-e)\cdot M$.
If now $F:\cat C\to \cat D$ is a $\kk$-linear
functor between such categories, it follows from the above description
that $(F(M))^S = e\cdot F(M)\simeq F(e\cdot M) = F(M^S)$, as long
as $F(M)$ is equipped with the obvious $S$-action
$S\to \Aut_{\cat C}(M)\to \Aut_{\cat D}(F(M))$ induced by~$F$.
\end{Rem}

\begin{Rem}
\label{rem:sym}%
The above ideas are applied below to the symmetric group~$S=\Spmo$
on $p-1$ letters acting on an object $M$ in a $\kk$-linear
category~$\cat C$, where $p>0$ is the
characteristic of~$\kk$. The three $\kk$-linear categories we use are
in turn the module category $\kk C_p\MMod$, the stable
category~$\kk C_p\SStab$ and
finally its quotient $\cat D=\kk C_p\SStab/\Rad$ by
the Kelly radical. The object $M$ with an action of~$S=\Spmo$
is $M=[i]\potimes{(p-1)}$ with action by permutation of the factors;
and we also consider the images of~$M$ under the quotient functors
$P:\kk C_p\MMod\to \kk C_p\SStab$ and
$Q:\kk C_p\SStab\to \cat D$. Since both functors are quotient
functors, \ie only change the morphisms,
the object in question remains the ``same" $[i]\potimes{(p-1)}$,
if one wishes. By Remark~\ref{rem:fixed}, its $S$-fixed sub-object~$M^S$
is preserved by the functors~$P$ and~$Q$.
\end{Rem}

\begin{proof}[Proof of Theorem~\ref{thm:C_p}]
When $p=2$, the category $\kk C_2\SStab$ is equivalent
to $\kk\MMod$. That is, the only non-projective indecomposable module is
the trivial module $\kk$ and every module is stably isomorphic to a
coproduct of trivial modules. See~\cite{CrawleyJonsson64} or~\cite{Warfield69}.
Thus, the theorem is trivially true in
this case. As a consequence,  we assume hereafter that $p$ is odd.

Consider the additive quotient of~$\kk C_p\SStab$ by its
Kelly radical:
\[
\xymatrix{
\kk C_p\SStab \ar@{->>}[r]^-{Q}
& \cat D:=\frac{\Displ \kk C_p\SStab }{\Displ \Rad(\kk C_p\SStab )}
}
\]
By Theorem~\ref{thm:rad} (or Example~\ref{exa:C_p^n}), this Kelly radical is a tensor-ideal.
This also uses the fact that every object of~$\kk C_p\SStab$ is a coproduct of finite-dimensional ones, see~\cite{CrawleyJonsson64,Warfield69} again.
Therefore the above functor~$Q$ is a
tensor functor. Hence, $B:=Q(A)$ is a separable commutative
ring in~$\cat D$.

The quotient category~$\cat D$ is actually abelian semisimple.
Indeed, in the tt-category $\kk C_p\SStab$ every object is a (possibly infinite) coproduct of
finite dimensional indecomposables and there are $p-1$ indecomposables up to isomorphism:
$[1]=\kk$, $[2]$, \ldots, $[p-1]$; furthermore for all $i\neq j$, we
have $\Rad([i],[j])=\Hom([i],[j])$ and we have
$\Hom([i],[i])/\Rad([i],[i])\simeq \kk$. This means that
the quotient
\[
\cat D \simeq \bigoplus_{i=1}^{p-1} \big(\kk\MMod)_i
\]
is a (co)product of copies of the category of $\kk$-vector
spaces, indexed by~$i=1,\ldots, p-1$. (Finite coproducts of additive categories coincide with their products.)
The subtlety about the quotient category~$\cat{D}$ comes from its tensor product, which is governed by Formula~\eqref{eq:i@j}.

Now, choose $1<i<p$ and suppose that $A$ has a copy of~$[i]$
among its direct summands (in~$\kk C_p\SStab$). Consider the
$\otimes$-power $M:=[i]\potimes{(p-1)}$ in~$\kk C_p\MMod$, with the obvious
action of the symmetric group~$S=\mathfrak{S}_{p-1}$ on
$p-1$ letters by permuting the factors as announced in
Remark~\ref{rem:sym}. Since $(p-1)!$ is invertible in~$\kk$,
we can describe the fixed subobject $M^S$ as in Remark~\ref{rem:fixed}.
Indeed, $(p-1)! = -1$ in $\kk$. So in other words, the
symmetric power of~$[i]$ equals
$S^{p-1}[i]=M^S=e\cdot M$ where $e= -\sum_{s\in\mathfrak{S}_{p-1}}s$.

By Remark~\ref{rem:sym}, we have $QP(M)^S\simeq QP(M^S)$
in~$\cat D$. However, by the work of Almkvist and
Fossum~\cite{AlmkvistFossum78} we have that $M^S=S^{p-1}[i]$ is
projective for $i\ge 2$, as we assume here. Hence, $P(M^S)\simeq 0$
in~$\kk C_p\SStab$ and therefore $QP(M)^S=0$ as well.
We can then apply Remark~\ref{rem:M^G-M_G} to the object $QP(M)$,
which is really just $[i]\potimes{(p-1)}$ but now viewed
in~$\cat D$. By that remark, the morphism
$\nu:QP([i]\potimes{p-1})\to Q(A\potimes{p-1})\otoo{Q\mu}Q(A)$
is zero, since it satisfies $\nu\circ s=\nu$ for all~$s\in\Spmo$
by commutativity of~$\mu$.

In summary, we have shown that every direct summand
$Q([i])$ of $B=Q(A)$ with $i>1$ has to be nilpotent in
the separable commutative ring-object~$B$ of~$\cat D$.
Let now $I\subseteq B=Q(A)$ be the ideal generated by all direct
summands $Q([i]) \subseteq Q(A)$ for $i>1$ in the semisimple abelian
category~$\cat D$. Because the category $\cat D$ is abelian and
semisimple, this ideal~$I$ consists of the sum of the images $\mu(U \otimes V)$ where $U$ is
any summand of $B$ and $V$ is any direct summand of $B$ that is
isomorphic to~$Q([i])$.
By the above discussion, this ideal~$I\subseteq B$ is nilpotent.
Hence, it must be zero by Proposition~\ref{prop:no-nil}\,(b).
This shows that $A$ has no direct summand isomorphic
to~$[i]$ for $i>1$.

Thus we have proved that $A$ is a $\kk C_p$-module with trivial
$C_p$-action, and it belongs to the image of the
fully faithful tensor functor
$\pi^*:\kk\MMod\hook \kk C_p\SStab$ (where $\pi:C_p\to 1$) and
we reduce again to the field case.
\end{proof}

\goodbreak
\section{The case of the general cyclic group}
\label{se:general}%
\medbreak

Let $\kk$ be a field of characteristic~$p>0$, and let $C_{p^n}$
be the cyclic group of order~$p^n$ for $n\ge 1$.
The following statement implies the theorem given in the Introduction:

\begin{Thm}
\label{thm:main}%
Let $A\in\kk C_{p^n}\sstab$ be a tt-ring in the stable module
category. Then there exists a commutative and separable ring-object $A$ in
$\kk C_{p^n}\mmod$ whose image in the stable category is~$A$.
Explicitly, if we assume~$\kk$ separably closed, there exist
a finite $C_{p^n}$-set~$X$ such that $A\simeq \kk X$, or equivalently
there exist subgroups $H_1,\ldots,H_r\le G$ such that
$A\simeq A^G_{H_1}\times\cdots \times A^G_{H_r}$ (see Example~\ref{exa:k(G/H)}).
\end{Thm}

We need a little preparation.

\begin{Prop}
\label{prop:res-dim}%
Suppose that $G = \langle g \rangle$ is a cyclic group of order $p^n>1$ and
$H = \langle g^{p^{n-m}} \rangle$ is the
subgroup of order $p^{m}>1$. Suppose that $M$
is a $\kk G$-module having no nonzero
projective summands and suppose that the restriction
$M_{\downarrow H}$ has a nonzero $\kk H$-projective summand.
Then $M_{\downarrow H}$ has an
indecomposable summand of dimension $p^{m}-1 = \vert H \vert -1$.
\end{Prop}

\begin{proof}
We may assume without loss of generality that $M$ is indecomposable. Let $r$
be the dimension of~$M$. Let $t = g-1$,
so that $\kk G \simeq \kk [t]/(t^{p^n})$.
Then $M$ has a basis $v_1 , \dots, v_r$ such that $tv_i = v_{i+1}$ for all
$i = 1, \dots, r-1$ and $tv_r = 0$. For $i>r$, set by convention $v_i:=0$.
The algebra of the subgroup
$H$ is~$\kk [y]/(y^{p^{m}})$ where
$y = g^{p^{n-m}}-1 = t^{p^{n-m}}$. As a $\kk H$-module $M_{\downarrow H}$ is
generated by $v_1, \dots, v_{p^{n-m}-1}$, and $yv_i =
v_{p^{n-m}+i}$. If $r<p^{n-m}$, then the $\kk H$-module
$M_{\downarrow H}$ would be trivial; hence $r \ge p^{n-m}$.
It is then a straightforward exercise to show that
\[
M_{\downarrow H} \ \simeq \
\kk Hv_1 \oplus \kk Hv_2 \oplus \dots \oplus \kk Hv_{p^{n-m}}.
\]
The fact that $M_{\downarrow H}$ has a projective direct summand
means that $y^{p^m-1}M \neq \{0\}$ and therefore
$r > (p^{m}-1)p^{n-m} = p^n -p^{n-m}$.
Because $M$ is not projective $r < p^{n}$. Now write
$r = (p^{n}-p^{n-m}) +s$ where $1 \leq s < p^{n-m}$. Then we have that
$y^{p^m -1}v_{s+1} = v_{r+1} = 0$ and $y^{p^m -2}v_{s+1}\neq 0$. Thus the submodule $\kk Hv_{s+1}$ is an
indecomposable direct summand of $M_{\downarrow H}$
having dimension $p^m-1=\vert H\vert -1$
as asserted.
\end{proof}

This proposition is quite useful except in the fringe case where
$p=2$ and $n=2$. This is the unique case in which $p^{n-1}-1$ equals~1, and we handle it separately.
\begin{Prop}
\label{prop:C_4}%
Let $A\in\kk C_4\sstab$ be a tt-ring over the cyclic group of
order four, with $\chara(\kk)=2$. Then $A$ has no indecomposable
summand $[3]$ of dimension~3.
\end{Prop}

\begin{proof}
Use the Notation $[i]$ of Example~\ref{exa:C_p^n} for $i=1,2,3$.
Note that $[1]=\unit$ and $[3]\simeq \Sigma\unit$, so that
$[3] \otimes [3] = [1]$ in the stable category. This follows by
the tensor formula for shifts. Direct inspection shows that
$[2]\otimes[2]\simeq [2]\oplus [2]$. It follows that there
is a well-defined $\otimes$-ideal $\CI$ of morphisms
in~$\cat C=\kk C_4\sstab$ which consists of those morphism
which factor via some $[2]^{\oplus m}$ for some $m\ge 1$.
(One can verify that this is the ideal $\Ratcat{C}$,
but this is not essential.) The additive quotient
$\cat C\onto \cat C/\CI$ amounts to quotienting out
all objects~$[2]^{\oplus m}$ for $m\ge 1$. The resulting
category~$\cat C/\CI$ consists of two copies of
$\kk$-vector spaces, one generated by the image of~$[1]$
and one by the image of~$[3]=\Sigma[1]$, and its tensor
product is forced by $[3]\otimes[3]\simeq [1]$.
In other words, $\cat C/\CI$ is the category of finite
dimensional $\bbZ/2$-graded $\kk$-vector spaces.
We saw in Proposition~\ref{prop:super} that a separable
commutative ring-object in that category must be
concentrated in degree zero, \ie the image of~$A$ in~$\cat C/\CI$
contains no copy of~$[3]$.
\end{proof}

\begin{Prop}
\label{prop:infl}%
Let $N\lhd G$ be a normal subgroup of a finite group~$G$ and
assume that $p$ divides the order of $N$.
Consider $\pi:G\onto \bar G=G/N$ the corresponding quotient.
Then inflation $\Infl^{\bar G}_G=\pi^*:\kk \bar G\mmod\to \kk G\sstab$,
from the \emph{ordinary} module category of~$\bar G$ to the \emph{stable} category of~$G$,
is fully faithful and its essential image consists of those
objects isomorphic in~$\kk G\sstab$ to some $\kk G$-module~$M$ such that
$\Res^G_N M$ has trivial $N$-action.
\end{Prop}

\begin{Rem}
Objects of $\kk G\sstab$ are the same as those of
$\kk G\mmod$, \ie finitely generated $\kk G$-modules.
However, the property ``$\Res^G_N M$ has trivial $N$-action"
is not stable under isomorphism in~$\kk G\sstab$,
since one can add to $M$ a projective $\kk G$-module.
This explains the phrasing of the above statement.
\end{Rem}

\begin{proof}[Proof of Proposition~\ref{prop:infl}]
The image of $\Infl^{\bar G}_G:\kk \bar G\mmod\to \kk G\mmod$
consists precisely of those $\kk G$-modules on which $N$ acts
trivially. This gives the statement about the essential image
by taking closure under isomorphism in~$\kk G\sstab$.
To show that $\Infl^{\bar G}_G:\kk \bar G \mmod \to \kk G\sstab$
is fully faithful, note first that it is full because both
$\Infl^{\bar G}_G:\kk\bar G\mmod\to \kk G\mmod$ and
$\kk G\mmod\onto\kk G\sstab$ are full. For faithfulness,
consider the commutative diagram
\[
\xymatrix{
\kk G\sstab \ar[d]_-{\Res^G_N}
& \kk G\mmod \ar@{->>}[l]
& \kk\bar G\mmod \ar[l]_-{\Infl^{\bar G}_G} \ar[d]^-{\textrm{faithf.}}
\\
\kk N\sstab
&& \kk \mmod \ar[ll]_-{\textrm{faithf.}}
}
\]
As the right-hand and bottom functors are faithful, so is the top composite.
\end{proof}

\begin{proof}[Proof of Theorem~\ref{thm:main}.]
Let $A$ be an indecomposable tt-ring in~$\kk G\sstab$, where
$G=C_{p^n}$. We  proceed by induction on~$n$. The case $n=1$
was settled in Theorem~\ref{thm:C_p} so we assume $n\ge2$.
We can choose a $\kk G$-module $M$ representing the
object~$A$ in~$\kk G\sstab$ and therefore assume that $M$
has no projective summand.

Consider $N=C_p\lhd G$ the unique cyclic subgroup of order~$p$.
We claim that $\Res^G_N M$ has trivial $C_p$-action (in $\kk C_p\mmod$).

Suppose first that $p=2$. We need to prove that $\Res^G_{C_2} M$
does not contain a projective summand. By Proposition~\ref{prop:C_4},
we know that $\Res^G_{C_4}M$ has no indecomposable summand~$[3]$
of dimension~3. By Proposition~\ref{prop:res-dim} for $m=2$,
we know therefore that $\Res^G_{C_4}M$ has no projective factor either.
Hence, $\Res^G_{C_4} M$ consists of a sum of copies of~$[1]$ and~$[2]$,
which both restrict to trivial modules over~$C_2$.

Suppose now that $p$ is odd. Then we can use Proposition~\ref{prop:res-dim}
directly from~$G$ to $C_p$ (\ie take $m=1$). We know that
$\Res^G_{C_p}M$ consists only of trivial $\kk C_p$-modules and
possibly projectives, since its class in the stable category
is trivial by Theorem~\ref{thm:C_p}. None of those indecomposable
summands have dimension~$p-1$. Hence, Proposition~\ref{prop:res-dim}
tells us that $\Res^G_{C_p} M$ has no projective summand.

So, we have proved the claim: $\Res^G_{C_p} M$ has no projective summand.
On the other hand, by Theorem~\ref{thm:C_p}, this $\Res^G_{C_p} M$
is trivial in the stable category. Combining both facts, we have
shown that $\Res^G_{C_p} M$ has trivial $\kk C_p$-action.
By Proposition~\ref{prop:infl} for our $N=C_p\lhd G$, it follows
that the image~$A$ of~$M$ in $\kk G\sstab$ belongs to the essential
image of the fully faithful tensor functor
$\kk \bar G\mmod\to \kk G\sstab$ where $\bar G=G/N$. By the much
easier ordinary module category case (see~\cite[Rem.\,4.6]{Balmer15}),
we see that $A\simeq \pi^*\kk X\simeq \kk \pi^*X$ for some finite
$\bar G$-set~$X$, which can then be viewed as a finite $G$-set $\pi^*X$
through $\pi:G\to \bar G$.
\end{proof}



\end{document}